\documentclass[12pt,amscd,amssymb]{article}																																					 
\usepackage[centertags]{amsmath}
\usepackage{amsmath}
\usepackage{mathrsfs}
\usepackage{latexsym}
\usepackage{amsthm}
\usepackage{amsfonts}
\usepackage{amssymb}
\usepackage{stmaryrd}

\usepackage{xspace}
\usepackage{verbatim}
\usepackage{enumerate}
\usepackage{indentfirst}
\usepackage{lipsum}
\makeatletter
\newcommand{\authorfootnotes}{\renewcommand\thefootnote{\@fnsymbol\c@footnote}}%
\makeatother

\theoremstyle{plain}

\DeclareMathOperator{\sll}{SL}
\DeclareMathOperator{\gl}{GL}
\DeclareMathOperator{\psl}{PSL}

\newtheorem{theorem}{Theorem}[section]

\newtheorem{fact}[theorem]{Fact}
\newtheorem*{theorem*}{Theorem}
\newtheorem*{maintheorem*}{Main Theorem}
\newtheorem*{lemma*}{Lemma}

\newtheorem*{intconj*}{Intermediate Conjecture}

\newtheorem*{prop*}{Proposition}
\newtheorem*{conj*}{Principal Conjecture}
\newtheorem*{thm3.1*}{Theorem 3.1}
\newtheorem*{thm3.3*}{Theorem 3.3}
\theoremstyle{definition}

\newtheorem{remark}[theorem]{Remark}
\newtheorem*{remark*}{Remark}
\newtheorem{case}{Case}
\newtheorem{step}{Step}
\newtheorem{claim}{Claim}
\newtheorem*{claim*}{Claim}

\begin{document}
\begin{center}
  \LARGE
A note on conjugacy problem for finite Sylow subgroups of infinite linear groups \par \bigskip

  \normalsize
 \authorfootnotes

 P\i nar U\u{G}URLU KOWALSKI \footnote{Correspondence: pinar.ugurlu@bilgi.edu.tr

 }

  Department of Mathematics, \.{I}stanbul Bilgi  University, \.{I}stanbul,  Turkey \par \bigskip

\end{center}


\begin{abstract}
We prove the conjugacy of Sylow $2$-subgroups in  pseudofinite $\mathfrak{M}_c$ (in particular linear) groups under the assumption that there is at least one finite Sylow $2$-subgroup.  We observe the importance of the pseudofiniteness assumption by analyzing  an example of a linear group with non-conjugate finite Sylow $2$-subgroups which was constructed by Platonov.
\end{abstract}

\section{Introduction}
Sylow theory was originated and developed in the world of finite groups. There is also some work on a possible generalization to infinite   groups (for a comprehensive survey see  \cite{algebra}). While in some particular families of infinite groups conjugacy results hold for Sylow subgroups, there are pathological situations (non-conjugate Sylow $p$-subgroups) even in the case of linear groups. However, existence of a finite Sylow $p$-subgroup yields conjugacy results in some classes of groups (e.g. groups of finite Morley rank for $p=2$ \cite[Lemma 6.6]{abc} and locally finite groups \cite[Proposition 2.2.3]{dixon}). In this paper, we show that this existence assumption gives the desired conjugacy result for Sylow $2$-subgroups in the case of  pseudofinite $\mathfrak{M}_c$-groups.  We also present an interesting example constructed by Platonov  \cite[Example 4.11]{platonov} which shows that having a finite Sylow $2$-subgroup does not guarantee conjugacy in the case of linear groups.

The main result of this paper is stated below.

\begin{thm3.3*}
If one of the Sylow $2$-subgroups of a pseudofinite  $\mathfrak{M}_c$-group $G$  is finite then all Sylow 2-subgroups of $G$ are conjugate and hence finite.
\end{thm3.3*}

The structure of this paper is as follows.

In the second section, we recall some of the basic notions in group theory   and we fix our terminology and notation.

In the  third section, we emphasize  some properties of  pseudofinite groups and provide some (non-)examples. Then we  state and prove  our main result (Theorem 3.3).

In the last section,  we analyze an example constructed by  Platonov \cite[Example 4.11]{platonov} in detail.

\section{Preliminaries}
In  this section, we recall definitions of some basic notions in group theory and list some well-known results which will be needed in the sequel.

 A group $G$ is called  \emph{periodic} if every element of it has finite order and  it is called a \emph{$p$-group} for a prime $p$, if each element of $G$ has order $p^n$ for a natural number  $n$. An example of an infinite $p$-group is the Pr\"{u}fer $p$-group, denoted by $C_{p^{\infty}}$. By a \emph{Sylow $p$-subgroup} of a group $G$,  we mean a maximal $p$-subgroup of $G$. Note that existence of  Sylow $p$-subgroups is guaranteed by Zorn's lemma.

A group $G$ is said to satisfy \emph{normalizer condition} if any proper subgroup is properly contained in its normalizer. It is well-known that finite nilpotent groups satisfy normalizer condition.

  Let $\mathcal{P}$ denote a group theoretical property such as solvability, nilpotency, commutativity, finiteness etc.  A group $G$ is called \emph{locally $\mathcal{P}$} if every finite subset of $G$ generates a subgroup with the property $\mathcal{P}$. A group $G$ is said to be \emph{$\mathcal{P}$-by-finite} if $G$ has a normal subgroup $N$ with the property $\mathcal{P}$ such that the quotient group $G/N$ is finite.

 A  \emph{linear group} is a subgroup $G \leqslant \gl_n(F)$  for some field $F$ where $\gl_n(F)$ denotes the general linear group over $F$.  A group $G$ is said to satisfy \emph{descending chain condition on centralizers} (or \emph{minimal condition on centralizers})  if every proper chain of centralizers in $G$ stabilizes after finitely many steps and such groups  are called \emph{$\mathfrak{M}_c$-groups}.  If moreover there is a global finite bound on the length of such chains, then we say that $G$ has \emph{finite centralizer dimension}. 
 It is well-known that any linear group  has finite centralizer dimension (see for example the remark after Corollary 2.10 in \cite{wehf}) and hence  the class of $\mathfrak{M}_c$-groups contain the class of linear groups.

The following results about $\mathfrak{M}_c$-groups which generalize the corresponding classical results for linear groups will be needed in the sequel.


\begin{fact} [Wagner,  Corollary 2.4 in \cite{wagner1}]  \label{wagner}
Sylow $2$-subgroups of $\mathfrak{M}_c$-groups  are locally finite.
\end{fact}

\begin{fact} [Bryant, Theorem A in \cite{bryant}]  \label{nf}
 Periodic, locally nilpotent $\mathfrak{M}_c$-groups  are  nilpotent-by-finite.
\end{fact}


\begin{fact}[Wagner, Fact 1.3 in \cite{wagner1}] \label{sol}
 A nilpotent-by-finite and locally nilpotent group has non-trivial center and satisfies the normalizer condition.
 \end{fact}




\begin{remark} \label{locfin}
 Note that Sylow $2$-subgroups of $\mathfrak{M}_c$-groups  are locally finite by Fact~\ref{wagner} and hence  locally nilpotent since finite $2$-groups are nilpotent. Therefore, they are nilpotent-by-finite by Fact~\ref{nf} and they satisfy the normalizer condition by Fact~\ref{sol}.
\end{remark}

\section{A conjugacy result for pseudofinite groups}
In this section, we briefly introduce    pseudofinite groups without giving precise definitions of the related notions (such as ultrafilters, ultraproducts and other basic model theoretical concepts) and we emphasize some properties of these groups which will be needed in the proof of the main result of this paper. We refer the reader to the books \cite{bells} and \cite{changk} for a detailed information about the ultraproduct construction, to \cite{ugurlu} for a more complete introduction to  pseudofinite groups and to \cite{wilson} for a more detailed discussion of these groups. 

 Pseudofinite groups are defined 
as infinite models of the theory of finite groups. These groups are group theoretical analogues of pseudofinite fields  which were introduced, studied and algebraically characterized by James Ax (see \cite{ax}). Unfortunately, such an algebraic characterization is not known for pseudofinite groups.

One can also describe (up to elementary equivalence) pseudofinite groups as  non-principal ultraproducts of finite groups  (see \cite{wilson} for details).  This description  together with \L o\'{s}'s Theorem \cite{los1} (which states that a first order formula is satisfied in the ultraproduct if and only if it is satisfied in the structures indexed by a set belonging to the ultrafilter)
allow us to logically characterize pseudofinite groups as infinite groups satisfying  the first order properties shared by \emph{almost all} (depending on the choice of an ultrafilter) of the finite groups.

A well-known example of a pseudofinite group is the additive group of the rational numbers $(\mathbb{Q}, +)$ (see e.g. \cite[Fact 2.2]{ugurlu}).  However, the additive group of integers, $(\mathbb{Z}, +)$, is not a pseudofinite group, since while all finite groups satisfy  the following first order statement
$$\text{the map $x\mapsto x+x$ is one-to-one if and only if it is onto},$$
 the group $(\mathbb{Z},+)$ does not.

In the following remark, we mention another first order property shared by all finite groups. This property will be an important ingredient of our proof.
\begin{remark}\label{inv}
In any finite group $G$, two involutions $g, h$ are either conjugate or there is an involution $y$ commuting with both $g$ and $h$. The first order sentence below shows that  this statement can be expressed in a  first order way in the language of groups.
$$\forall g,h \ \ \left[(g\neq 1\neq h) \wedge  (g^2=1=h^2)\right] \ \longrightarrow$$
$$\left[\left(\exists x\ g^x=h)\vee (\exists y\ (y\neq 1) \wedge (y^2=1) \wedge (g^y=g) \wedge (h^y=h)\right)\right].$$
Since this property is satisfied by all finite groups,  pseudofinite groups satisfy it as well.
\end{remark}

Although pseudofinite groups are in a way similar to finite groups, there are also many differences. For example, while all finite groups are isomorphic to linear groups, this is not true for pseudofinite groups. To see this, it is enough to construct a pseudofinite group which does not have  finite centralizer dimension since all linear groups have finite centralizer dimension.  Consider a non-principal ultraproduct of alternating groups,  $G=\prod A_n/\mathcal{U}$, such that there is no bound on the orders of the alternating groups in the ultraproduct (if there is a bound, then the ultraproduct is finite, that is, $G$ is not a pseudofinite group). By just considering centralizers of disjoint even number of transpositions, it easy to see that the centralizer dimension of the  alternating groups increases as the rank increases. Since having finite centralizer dimension $c$ is a first order property of groups (see \cite{duncan}), the ultraproduct $\prod A_n/\mathcal{U}$ has finite centralizer dimension  if only if there is a bound on the orders of the alternating groups in the ultraproduct. However, by our assumption there is no bound on the orders of the alternating groups. This proves that $\prod A_n/\mathcal{U}$ does not have finite centralizer dimension, and hence it is not linear.

We will need the following  result  about  pseudofinite groups.


\begin{fact} [Houcine and Point, Lemma 2.16 in \cite{point1}] \label{point1} 
 Let $G$ be a pseudofinite group. Any definable subgroup or any quotient by a
definable normal subgroup is (pseudo)finite.
\end{fact}

Note that when we say \emph{definable} we mean definable in the language of groups and possibly with parameters (for details see for example the book \cite{changk}). In particular finite sets are definable. It is well-known that if $X$ is a definable set in a group $G$ then the centralizer and the normalizer of $X$ in $G$ are definable. Moreover, if $G$ is a group of finite centralizer dimension then the centralizer of any set in $G$ is definable.

In the proof of the following proposition, we will use the well-known results about definability mentioned above as well as some ideas from the proof of a similar result in the context of groups of finite Morley rank (see Lemma 6.6 in \cite{abc}).
\begin{theorem}\label{main}
If one of the Sylow $2$-subgroups of a pseudofinite  $\mathfrak{M}_c$-group $G$  is finite then all Sylow $2$-subgroups of $G$ are conjugate and hence finite.
\end{theorem}
\begin{proof}
Let $P$ be a finite Sylow $2$-subgroup of $G$. Assume that there is a  Sylow $2$-subgroup $Q$ of $G$ which is not  conjugate to $P$ and let $D=P\cap Q$. Without loss of generality, we may assume that $Q$ is chosen so that  $|D|$ is maximal (for this fixed finite  Sylow $2$-subgroup $P$).  Since $D$ is finite, both $D$ and $N_G(D)$ are definable and hence $N_G(D)/D$ is a (pseudo)finite group  by Fact~\ref{point1}.  Moreover, as both $P$ and $Q$ are locally finite (Fact~\ref{wagner}), locally nilpotent and nilpotent-by-finite, they satisfy the normalizer condition (see Remark~\ref{locfin}).  Therefore,  we have $$N_G(D)\cap P=N_P(D) > D \ \  \text{and} \ \  N_G(D)\cap Q=N_Q(D) > D.$$
\begin{claim*}
There are non-conjugate Sylow $2$-subgroups $P_1, Q_1$ of $G$ such that $|P_1 \cap Q_1|> |P \cap Q|$, and $P$ and $P_1$ are conjugate.
\end{claim*}
\begin{proof}[Proof of the claim]
  Take two involutions $\bar{i}=i D, \bar{j}= j D$ from the nontrivial $2$-subgroups $N_P(D)/D$ and $N_Q(D)/D$ of $N_G(D)/D$.  We know that they are either conjugate or commute with another  involution in $N_G(D)/D$ since  $N_G(D)/D$ is (pseudo)-finite (see Remark~\ref{inv}).
\begin{case}  Assume that $\bar{i}, \bar{j}$ are conjugate in $N_G(D)/D$.
 \end{case}
In this case, we have $\bar{i}^{\bar{x}}=\bar{j}$ for some $\bar{x} \in N_G(D)/D$. This means that $xix^{-1}D=j D$, that is, $j=xix^{-1}d$ for some $d\in D$. We get $j=xix^{-1}d x x^{-1}=x i d_1 x^{-1}$ for some $d_1\in D$. But since $i d_1 \in P$, we get $j\in P^x \cap Q$. Moreover, since $x$ normalizes $D$, we get $D\leqslant P^x$ and hence $D \leqslant P^x \cap Q$. So,  we have $D < \langle D,j \rangle \leqslant P^x \cap Q$.  So take $P_1=P^x$ and $Q_1=Q$.

  \begin{case} Assume that there is an involution $\bar{k} \in N_G(D)/D$ such that  $\bar{i}$ and $\bar{j}$ commute with $\bar{k}$.
   \end{case}
  Now, consider the $2$-groups $\langle D,i,k \rangle$ and $\langle D, j, k\rangle$ and let $R_i$ and $R_j$ denote the Sylow $2$-subgroups of $G$ containing them respectively. Clearly we have the following inclusions:
 $$D < \langle D,i \rangle \leqslant P \cap R_i, \ \ D < \langle D,k \rangle \leqslant R_i \cap R_j,  \ \ D < \langle D,j \rangle \leqslant R_j \cap Q.$$ If $P$ is not conjugate to $R_i$  then take $P_1=P$ and $Q_1=R_i$. If $P$ is conjugate to $R_i$ but not conjugate to $R_j$ then take $P_1=R_i$ and $Q_1=R_j$. If $P$ is conjugate to both $R_i$ and $R_j$ then take $P_1=R_j$ and $Q_1=Q$.
\vspace{0.2 cm}

  So, the claim follows.
  \end{proof}
Let $P_1, Q_1$ be non-conjugate Sylow $2$-subgroups of $G$ as in the claim so that $P_1=P^g$ for some $g\in G$. Now, we have $$|D|=|P \cap Q| <|P_1 \cap Q_1|=|P^g \cap Q_1|=|(P^g \cap Q_1)^{g^{-1}}|=|P \cap Q_1^{g^{-1}}|.$$
Clearly, $Q_1^{g^{-1}}$ is a Sylow $2$-subgroup of $G$. By the maximality of $|D|$, we conclude that $P$ is conjugate to $Q_1^{g^{-1}}$ and hence conjugate to $Q_1$. This contradicts to the fact that $P_1=P^g$ is not conjugate to $Q_1$.
\end{proof}
\begin{remark}
The situation is quite complicated when we remove the assumption (in Theorem~\ref{main}) on the existence of a finite Sylow $2$-subgroup, even in the linear case (work in progress).

When we restrict ourselves to $\gl_n(K)$, there is a criterion given by Vol'vachev (for an arbitrary field $K$)  about the conjugacy of the Sylow $p$-subgroups (see \cite{volvachev2}). This criterion implies in particular that non-conjugacy can occur only for Sylow $2$-subgroups and only when the characteristic of the field $K$ is zero.
\end{remark}
\section{On Platonov's example}
In this section, we analyze  in detail an example constructed by Platonov  (Example 4.11 in \cite{platonov}). The reason for the detailed  presentation below is the fact that  some computational arguments  are skipped in Platonov's original article \cite{platonov}.


For each $i\in \mathbb{N}$, consider the following elements in the group $\sll_2(\mathbb{Q})$
$$g_i=\left(\begin{array}{cc}
          0& - p_i\\
          p_i^{-1}& 0\\
        \end{array}\right),
$$
where  $(p_i)_{i\in \mathbb{N}}$ is a sequence of distinct primes of the form $4k+3$, $k\in \mathbb{N}$.

We will observe that $S_i:=\langle g_i \rangle$ is a Sylow $2$-subgroup of $\sll_2(\mathbb{Q})$ of order $4$ for each $i$, however,  $S_i$ is not conjugate to $S_j$ if $i\neq j$.

Clearly, for each $i$ we have $|S_i|=4$.
\begin{claim}\label{cl1}
For each $i$, the group $S_i$ is a Sylow $2$-subgroup of $\sll_2(\mathbb{Q})$.
\end{claim}
Since a proof for this claim is not provided in \cite{platonov},  we list some  properties of $\sll_2(\mathbb{Q})$ (some of which are very well-known) which lead to a proof of  Claim~\ref{cl1}.

\begin{enumerate}[(1)]
  \item If $A \in \sll_2(\mathbb{Q})$ has finite order then $A$ is diagonalizable over $\mathbb{C}$.
  \item The group $\sll_2(\mathbb{Q})$ has a unique involution.
  \item There is no element of order $8$ in  $\sll_2(\mathbb{Q})$.
  \item \label{claim4} There is no subgroup of  order $8$ in $\sll_2(\mathbb{Q})$.
  \item \label{claim5} Sylow $2$-subgroups of $\sll_2(\mathbb{Q})$  are finite.
\end{enumerate}

Note that the properties $(1) - (3)$ follow from basic results in linear algebra. However, since $(4)$ and $(5)$ are more involved, we would like to support  them with proofs.

\begin{proof}[Proof of $(4)$.]
Assume that $H \leqslant \sll_2(\mathbb{Q})$ such that $|H|=8$. We know that there are only five groups of order $8$ up to isomorphism: $C_8$,  \  $C_2\times C_2\times C_2$,   \  $C_4 \times C_2$, \ $D_8$ \ (Dihedral group of order \ $8$), \ $Q_8$ \ (Quaternions). Since $\sll_2(\mathbb{Q})$ has a unique involution, we have $H \cong Q_8$ or $H \cong C_8$. However, as $\sll_2(\mathbb{Q})$ has no element of order $8$, the latter is not possible and hence $H \cong Q_8$.

Now, we will show that $Q_8$ does not embed in $\sll_2(\mathbb{Q})$ (actually, we can prove more: $Q_8$ does not embed in $\gl_2(\mathbb{R})$). By the structure of $Q_8$, it is enough to show that there are no $A, B \in \gl_2(\mathbb{R})$ such that \begin{equation} A^2=B^2=-I_2 \ \ \text{and} \ \ AB=-BA. \tag{$*$}
\end{equation}
We will observe $(*)$ in two steps.
\begin{step}
If $A \in \gl_2(\mathbb{R})$ is an element of order $4$ then there is $g\in \gl_2(\mathbb{R})$ such that $$A^g= \left(
        \begin{array}{cc}
          0 & 1 \\
         -1 & 0 \\
        \end{array}
      \right).$$
\end{step}
There is $h\in \gl_2(\mathbb{C})$ such that $A^h=\left(
        \begin{array}{cc}
         \lambda_1 & 0 \\
         0 & \lambda_2 \\
        \end{array}
      \right)$ where $\lambda_1^4=\lambda_2^4=1$. Without loss of generality we may assume that $\lambda_1$ is a primitive $4$th root of unity, that is, $\lambda_1=\pm i$.

      First assume $\lambda_1 = i$ and   let $\vec{z}=\vec{x}+\vec{y} i$ be a corresponding eigenvector (note that $\vec{x}, \vec{y} \in \mathbb{R}^2$). We have
\begin{align*}
 A \vec{z} &=  i \vec{z} \\
 A (\vec{x}+ \vec{y} i) &=  i (\vec{x}+ \vec{y} i) \\
 A \vec{x}+ A \vec{y} i &=  i \vec{x} - \vec{y}.
\end{align*}
Therefore, we get $ A \vec{x} = - \vec{y}$ and $A \vec{y} =  \vec{x}$. Note that $\{\vec{x}, \vec{y}\}$ forms a basis for $\mathbb{R}^2$ since they are linearly independent over $\mathbb{R}$ (If $\vec y=\alpha \vec x$ for some $\alpha \in \mathbb{R}$, then we get $A \vec{x} = - \alpha \vec{x}$ and $A \alpha \vec{x} =  \vec{x}$, which in turn gives $\alpha^2= -1$; a contradiction). When we represent $A$ with respect to this basis, we can conclude that $A$ is conjugate to the matrix $\left(
        \begin{array}{cc}
          0 & 1 \\
         -1 & 0 \\
        \end{array}
      \right)$. Similarly, if $\lambda_1 = -i$, then $A$ is conjugate to $\left(
        \begin{array}{cc}
          0 & -1 \\
         1 & 0 \\
        \end{array}
      \right)$ in $\gl_2(\mathbb{R})$ which is in turn conjugate to $\left(
        \begin{array}{cc}
          0 & 1 \\
         -1 & 0 \\
        \end{array}
      \right)$.
\begin{step}
There are  no  $A, B \in \gl_2(\mathbb{R})$ satisfying the conditions ($*$).
\end{step}
Assume that there are $A, B \in \gl_2(\mathbb{R})$ satisfying the conditions ($*$).
Using Step 1, without loss of generality, we can assume that   $A= \left(
        \begin{array}{cc}
          0 & 1 \\
         -1 & 0 \\
        \end{array}
      \right).$
Let $B=\left(
        \begin{array}{cc}
          a & b \\
         c & d \\
        \end{array}
      \right)$. Since $AB=-BA$,  we get $B=\left(
        \begin{array}{cc}
          a & b \\
         b & -a \\
        \end{array}
      \right)$. But on the other hand, since the order of $B$ is $4$, its square, $B^2 = \left(
        \begin{array}{cc}
          a^2+b^2 & 0 \\
         0 & a^2+b^2 \\
        \end{array}
      \right)$ is an involution. By the uniqueness of the involution, we get $$\left(
        \begin{array}{cc}
          a^2+b^2 & 0 \\
         0 & a^2+b^2 \\
        \end{array}
      \right)= \left(
        \begin{array}{cc}
          -1 & 0 \\
         0 & -1\\
        \end{array}
      \right),$$ which leads to a contradiction  since $a, b \in \mathbb{R}$.
      \end{proof}

\begin{proof}[Proof of $(5)$.]
Assume that $\sll_2(\mathbb{Q})$ has an infinite Sylow $2$-subgroup $P$. Since periodic linear groups are locally finite (see \cite{schur1911ueber}), $P$ is locally finite. Therefore, $P$ has a finite subgroup, say $X$, of order greater than $8$ (just consider the subgroup generated by $8$ distinct elements of $P$). Then, by Sylow's first theorem, $X$ has a subgroup of order $8$, which is clearly a subgroup of $\sll_2(\mathbb{Q})$. Since this is not possible, $(5)$ follows.
\end{proof}

By properties $(4)$ and $(5)$, we  conclude  that the groups $S_i=\langle g_i\rangle$ defined above are Sylow $2$-subgroups of $\sll_2(\mathbb{Q})$.

\begin{claim}
For $i\neq j$, the groups $S_i, S_j$ are not conjugate in $\sll_2(\mathbb{Q})$.
\end{claim}

Assume  that $S_i=S_j^g$ for some $g= \left(
        \begin{array}{cc}
          a & b \\
         c & d\\
        \end{array}
      \right) \in \sll_2(\mathbb{Q})$.  Then either we have $ g g_i g^{-1}=  g_j$ or $g g_i g^{-1}=  g_j^3$. We consider the first case.   \begin{align*}
 g g_i g^{-1}&=  g_j  \\
 g g_i &=  g_j g \\
  \left(
        \begin{array}{cc}
          a & b \\
         c & d\\
        \end{array}
      \right) \left(
        \begin{array}{cc}
          0 & - p_i \\
          p_i^{-1} & 0 \\
        \end{array}
      \right)&=  \left(
        \begin{array}{cc}
          0 & - p_j \\
          p_j^{-1} & 0 \\
        \end{array}
      \right) \left(
        \begin{array}{cc}
          a & b \\
         c & d\\
        \end{array}
      \right).
      \end{align*}
 This equation gives the following equalities:
 $$bp_i^{-1}=-cp_j, \ \ dp_i^{-1}=ap_j^{-1}, \ \ -ap_i=-dp_j, \ \ -cp_i=bp_j^{-1}.$$
Multiplying both sides of the first and third equations by $-c p_i$ and $-a p_j^{-1}$ respectively we get $$-bc=c^2p_j p_i, \ \ ad=a^2p_ip_j^{-1}.$$ By combining with the fact that  $ad-bc=1$, we have $$a^2p_ip_j^{-1}+c^2p_j p_i=1,$$ which in turn gives \begin{equation}
p_i(a^2+c^2p_j^2)=p_j.\tag{$\diamond$}
\end{equation}

We give an argument  for the impossibility of ($\diamond$), since it is skipped in \cite{platonov}. Note that the $p_j$-adic valuation of the right-hand side of the equation ($\diamond$) is clearly $1$. However, the $p_j$-adic valuation of the left-hand side is even by the following fact which is a folklore.

\begin{fact}
Suppose $p$ is a prime such that $p\equiv 3 (\text{mod}~4)$ and $\alpha, \beta \in \mathbb{Q}$. Then $v_p(\alpha^2+\beta^2)$ is even, where $v_p$ denotes the $p$-adic valuation.
\end{fact}
\begin{proof}
Firstly, without loss of generality we may assume that $\alpha, \beta$ are integers. To see this let $\alpha=\dfrac{\alpha_1}{\alpha_2}$, $\beta=\dfrac{\beta_1}{\beta_2}$ for some $\alpha_1, \alpha_2, \beta_1, \beta_2 \in \mathbb{Z}$. Clearly, we have $$v_p(\alpha^2+\beta^2)=v_p\left(\frac{\alpha_1^2}{\alpha_2^2}+\frac{\beta_1^2}{\beta_2^2}\right)=v_p((\alpha_1\beta_2)^2+(\beta_1\alpha_2)^2)-v_p((\alpha_2\beta_2)^2).$$ Therefore, $v_p(\alpha^2+\beta^2)$ is  even if and only if $v_p((\alpha_1\beta_2)^2+(\beta_1\alpha_2)^2)$ is even, since the valuation of a square is always even.

Since $p\equiv 3 (\text{mod}~4)$, $p$ is irreducible in the ring of Gaussian integers $\mathbb{Z}[i]$. (If $p=(x+yi)(z+ti)$ in $\mathbb{Z}[i]$, then squaring the norms of both sides we get $p^2=(x^2+y^2)(z^2+t^2)$. Since $p$ can not be sum of two squares (which is $0,1$ modulo $4$), either $x+yi$ or $z+ti$ is a unit, that is, $p$ is irreducible.). Moreover, since $\mathbb{Z}[i]$ is a unique factorization domain $p$ is a prime in $\mathbb{Z}[i]$. Now, if $p$ does not divide $ \alpha^2+\beta^2$ clearly the $p$-adic valuation is zero. Assume $p$ divides $\alpha^2+\beta^2$ (in $\mathbb{Z}$). Then $p$ divides $(\alpha + \beta i)(\alpha - \beta i)$ in $\mathbb{Z}[i]$. Then, as a prime in $\mathbb{Z}[i]$, $p$ divides either    $\alpha + \beta i$ or $\alpha - \beta i$. But then both $\alpha$ and $\beta$ are divisible by $p$  in $\mathbb{Z}$ and so  $\alpha^2=p^2 k^2$ and  $\beta^2=p^2 l^2$ for some $k,l \in \mathbb{Z}$. So now, $\alpha^2+\beta^2 = p^2 (k^2+l^2)$. Inductively, we can conclude that $v_p(\alpha^2+\beta^2)$ is even.
\end{proof}

 The computations for the second case (when $g g_i g^{-1}=  g_j^3$) are very similar and we obtain $$p_i(a^2+c^2p_j^2)= -p_j.$$ This is clearly not possible. We conclude that $S_i$ and $S_j$ are not conjugate.

 As a result, we have observed that  $\sll_2(\mathbb{Q})$ has infinitely many pairwise non-conjugate Sylow $2$-subgroups of order $4$.

\begin{remark}
Note that Theorem~\ref{main} together with Platonov's example show that $\sll_2(\mathbb{Q})$ is not a pseudofinite group. One may also observe this directly by first defining $\mathbb{Q}^{\times}$, the multiplicative group of $\mathbb{Q}$,  as a centralizer in $\sll_2(\mathbb{Q})$ and then using the fact that $\mathbb{Q}^{\times}$ is not a pseudofinite group (note that the definable  endomorphism of $\mathbb{Q}^{\times}$ which maps   $x$ to $x^3$  is one-to-one but not onto).

More generally, one can observe that for any infinite field $K$, the group $\sll_n(K)$ is pseudofinite if and only if $K$ is a pseudofinite field (for $n> 1$).  It is well-known that  if $K$ is a pseudofinite field then $\sll_n(K)$ is a pseudofinite group. To see the other direction, we work with the Chevalley group $\psl_n(K)$ which is also pseudofinite as a definable quotient of $\sll_n(K)$ and we will refer to the article of Wilson \cite{Wil}. In this article, Wilson states the following results obtained by S. Thomas in his dissertation \cite{thomas}: The class $\{X(K) ~|~ K \ \text{field}\}$ is an elementary class where $X$ denotes a Chevalley group of untwisted type and moreover whenever $X(K) \equiv X(F)$ and $K$ is pseudofinite it follows that $F$ is pseudofinite. If we apply Wilson's theorem  \cite[Theorem on pg. 471]{Wil}, together with the result of Thomas to $\psl_n(K)$, we get $\psl_n(K)\cong PSL_n(F)$ for some pseudofinite field $F$. By referring to the morevover part of Thomas' result, we finally conclude that the field $K$ is pseudofinite.
\end{remark}

\noindent
{\it Acknowledgements}. \  I would like to thank Alexandre Borovik for suggesting this topic to me. I also would like to thank to the referees for their valuable comments and suggestions.

\bibliographystyle{acm}
\bibliography{ref}
\end{document}